\documentclass[a4paper, 12pt]{amsart}

\usepackage[english]{babel}

\usepackage[latin1]{inputenc}
\usepackage[T1]{fontenc}
\usepackage{lmodern}

\usepackage{amsmath}
\usepackage{amssymb}

\usepackage{pstricks, pst-node, graphics}
\newpsobject{showgrid}{psgrid}{subgriddiv=1,griddots=10,gridlabels=6pt}

\usepackage{hyperref}

\usepackage{enumerate}

\renewcommand{\phi}{\varphi}
\renewcommand{\epsilon}{\varepsilon}

\renewcommand{\restriction}[0]{|}

\newcommand{\Size}[1]{\left\lvert #1 \right\rvert}

\newcommand{\Span}[1]{\left\langle\, #1 \,\right\rangle}

\newcommand{\Set}[1]{\left\{ #1 \right\}}

\newcommand{\Hc}[0]{\mathcal{H}}
\newcommand{\Jc}[0]{\mathcal{J}}
\newcommand{\Ic}[0]{\mathcal{I}}

\newcommand{\norm}[0]{\trianglelefteq}

\renewcommand{\theta}[0]{\vartheta}
\renewcommand{\phi}[0]{\varphi}

\newcommand{\F}{\mathbb{F}}

\DeclareMathOperator{\PSL}{PSL}
\DeclareMathOperator{\PGL}{PGL}

\DeclareMathOperator{\Aut}{Aut}

\DeclareMathOperator{\Inn}{Inn}

\DeclareMathOperator{\Hol}{Hol}

\DeclareMathOperator{\NHol}{NHol}
\DeclareMathOperator{\Out}{Out}

\DeclareMathOperator{\inv}{inv}

\newtheorem{dummy}{Dummy}
\numberwithin{dummy}{section}
\numberwithin{figure}{section}

\newtheorem{theorem}[dummy]{Theorem}

\newtheorem{remark}[dummy]{Remark}

\newtheorem{lemma}[dummy]{Lemma}

\newtheorem{proposition}[dummy]{Proposition}

\theoremstyle{definition}
\newtheorem{definition}[dummy]{Definition}
\newtheorem{notation}[dummy]{Notation}
\newtheorem{example}[dummy]{Example}

\theoremstyle{remark}

\begin{document}

\date{13 October 2017, 11:45 CEST --- Version 9.08%
}

\title[The holomorph of a finite perfect group]%
      {Groups that have the same holomorph as\\
       a finite perfect group}
      
\author{A. Caranti and F.~Dalla Volta}

\address[A.~Caranti]%
 {Dipartimento di Matematica\\
  Universit\`a degli Studi di Trento\\
  via Sommarive 14\\
  I-38123 Trento\\
  Italy} 

\email{andrea.caranti@unitn.it} 

\urladdr{http://science.unitn.it/$\sim$caranti/}


\address[F.~Dalla Volta]{Dipartimento di Matematica e Applicazioni\\
  Edificio U5\\
  Universit\`a degli Studi di Milano--Bicocca\\
  Via R.~Cozzi, 53\\
  I-20126 Milano\\
  Italy}

\email{francesca.dallavolta@unimib.it}

\urladdr{http://www.matapp.unimib.it/$\sim$dallavolta/}



\begin{abstract}    
 We describe  the groups  that have  the same  holomorph as  a finite
 perfect group. Our results are complete for centerless groups.

 When  the center  is  non-trivial, some  questions  remain open.  The
 peculiarities  of the  general case  are illustrated  by a  couple of
 examples that might be of independent interest.
\end{abstract}

\thanks{The authors are members of INdAM---GNSAGA. The first author
gratefully acknowledges support from the Department of Mathematics of
the University of Trento.}

\maketitle

\thispagestyle{empty}

\bibliographystyle{amsplain}

\section{Introduction}

We are concerned  with the question, when do two  groups have the same
holomorph?  Recall  that the  \emph{holomorph} of a  group $G$  is the
natural  semidirect product  $\Aut(G) G$  of $G$  by its  automorphism
group $\Aut(G)$. To put this problem in proper context, recall that if
$\rho : G \to S(G)$ is  the right regular representation of $G$, where
$S(G)$  is   the  group   of  permutations  on   the  set   $G$,  then
$N_{S(G)}(\rho(G)) =  \Aut(G) \rho(G)$ is isomorphic  to the holomorph
of $G$. We will also refer  to $N_{S(G)}(\rho(G))$ as the holomorph of
$G$, and  write it as  $\Hol(G)$. More  generally, if  $N \le
S(G)$ is a  regular subgroup, then $N_{S(G)}(N)$ is  isomorphic to the
holomorph of $N$.  We therefore begin to make the  above question more
precise by  asking for which  regular subgroup  $N$ of $S(G)$  one has
$N_{S(G)}(N) = \Hol(G)$.

W.H.~Mills has noted~\cite{Mills-multi} that such an $N$ need not be
isomorphic   to    $G$   (see   Example~\ref{ex:notiso},    but   also
Example~\ref{ex:non-iso} below, and the comment following it). In this
paper, we  will be interested in determining the following set, and
some naturally related ones
\begin{definition}\label{def:same}
  \begin{equation*}
    \Hc(G)
    =
    \Set{ N \le S(G) : \text{$N$ is regular, $N \cong G$ and $N_{S(G)}(N)
        = \Hol(G)$} }.
  \end{equation*}
\end{definition}
G.A.~Miller has  shown~\cite{Miller-multi} that the so-called
\emph{multiple holomorph} of $G$ 
\begin{equation*}
  \NHol(G) = N_{S(G)}(\Hol(G))
\end{equation*}
acts  transitively on $\Hc(G)$, and  thus the group
\begin{equation*}
  T(G)  = \NHol(G)  / \Hol(G)
\end{equation*}
acts regularly  on $\Hc(G)$.

Recently T.~Kohl  has described~\cite{Kohl-multi} the  set $\Hc(G)$
and the group $T(G)$
for $G$ dihedral or generalized quaternion.
In~\cite{fgag}, we have redone,  via a commutative
ring connection, the work of
Mills~\cite{Mills-multi}, which determined $\Hc(G)$
and $T(G)$ for $G$ a finitely generated abelian group.

In  this   paper  we  consider  the   case  when  $G$  is   a  finite,
\emph{perfect}  group,  that  is,  $G$  equals  its  derived  subgroup
$G'$. 

If
$G$  has also  trivial center,  then  one can  show that  if $N  \norm
\Hol(G)$ is a regular subgroup, then $N \in \Hc(G)$ (in particular, $N
\cong G$).  The elements of  $\Hc(G)$ can be  described in terms  of a
Krull-Remak-Schmidt decomposition of $G$ as  a group with $\Aut(G)$ as
group  of  operators, that is, in terms of the unique decomposition of
$G$ as a direct product of non-trivial characteristic subgroups that
are indecomposable as the direct product of characteristic
subgroups. The group $T(G)$ turns out to be an elementary abelian
$2$-group. 

If $G$ has non-trivial center, the  regular subgroups $N$ such that $N
\norm \Hol(G)$ can still be described in terms of the decomposition of
$G$  as the  central product  of non-trivial,  perfect, characteristic
subgroups,   that  are   indecomposable  as   a  central   product  of
characteristic subgroups. 

However,   these   $N$   need   not  be   isomorphic   to   $G$   (see
Example~\ref{ex:non-iso} below, and  the   comment  following  it),   and  the
structure of $T(G)$ in this case is not clear to us at the moment. The
difficulties  here are  illustrated by  the following  examples, which
might be of independent interest.
\begin{example}\label{ex:non-iso}
  There  is  a  group  $G$  which   is  the  central  product  of  two
  characteristic  subgroups  $Q_{1},  Q_{2}$,  such that  $G$  is  not
  isomorphic to the group $(G, \circ)$ obtained from $G$ by replacing
  $Q_{1}$ with its opposite.
\end{example}
Recall that the opposite of a group $Q$ is the group obtained by
exchanging the order of factors in the product of $Q$.

We will see  in Section~\ref{sec:main} that the group  $(G, \circ)$ is
isomorphic to  a regular subgroup  $N$ of $S(G)$ such that  $N_{S(G)}(N) =
\Hol(G)$. Therefore in our context the latter condition does not imply $N
\cong G$.
\begin{example}\label{ex:tba}
  There  is  a  group  $G$  which is  the  central  product  of  three
  characteristic subgroups  $Q_{1}, Q_{2},  Q_{3}$ such that  $Q_{1}$ and
  $Q_{2}$ are  not characteristic  in the group  obtained from  $G$ by
  replacing $Q_{1}$ with its opposite.
\end{example}
As we will see in Section~\ref{sec:main}, this example shows that if
$N \norm \Hol(G)$ is a regular subgroup, for $G$ perfect, we may well have that
$N_{S(G)}(N)$ properly contains $\Hol(G)$. However, if $\Aut(G)$ and
$\Aut(N)$ have the same order, then $N_{S(G)}(N) = \Hol(G)$ (see
Lemma~\ref{lemma:reghol}\eqref{item:reghol-2}). 

Example~\ref{ex:non-iso}~and     \ref{ex:tba}     are     given     in
Subsection~\ref{sub:centermore} as  Proposition~\ref{prop:non-iso} and
Proposition~\ref{prop:not-char}.

The      plan     of     the      paper     is      the     following.
Sections~\ref{sec:holomorph}~and   \ref{sec:same-holomorph}  introduce
the holomorph and  the multiple holomorph.  In these  sections, and in
the  following ones,  we have  chosen  to repeat  some elementary  and
well-known arguments, when we have  deemed them handy for later usage.
In  Sections~\ref{sec:regular}~and \ref{sec:normal-regular} we  give a
description of  the regular subgroups  $N$ of $\Hol(G)$, and  of those
that are  normal in $\Hol(G)$, in terms  of a certain map  $\gamma : G
\to \Aut(G)$. This leads to a group operation $\circ$ on $G$ such that
$N$ is isomorphic to $(G, \circ)$. In Section~\ref{sec:commutators} we
show   that  the  values   of  $\gamma$   on  commutators   are  inner
automorphisms, and this leads us to consider perfect groups.

In Section~\ref{sec:main} we study the case of finite, perfect groups.
We first  obtain a  description of the  normal subgroups  of $\Hol(G)$
that are  regular, in terms of certain  central product decompositions
of  $G$   (Theorem~\ref{thm:cpd}). We then discuss separately, as
explained above, the centerless case, where we can give a full
picture, and the general case, where some questions remain
open. Section~\ref{sec:After-Holt} deals with a
representation-theoretic method that is critical  for the
construction of the examples.

We note, as in~\cite{Kohl-multi}, that this work is related to the
enumeration of Hopf-Galois structures on separable field
extensions, as C.~Greither and B.~Pareigis have
shown~\cite{GrePar} that these structures can be described through the
regular subgroups of a suitable symmetric group, which are normalized
by a given regular subgroup; this connection is exploited 
in the work of L.~Childs~\cite{Chi},
N.P.~Byott~\cite{Byo}, and Byott and Childs~\cite{ByoChi}.

Our discussion of (normal) regular subgroups touches also on the
subject of skew braces \cite{bra}, see Remark~\ref{rem:braces}.

We  are  very   grateful  to  Robert  Guralnick   for  several  useful
conversations. We are indebted to  Derek Holt for kindly explaining to
us in careful detail the example and the construction method which led
to Proposition~\ref{lemma:After-Holt}.

\section{The holomorph of a group}
\label{sec:holomorph}

\begin{notation}
  We write permutations as exponents, and denote compositions
  of maps by juxtaposition. We compose maps  left-to-right.
\end{notation}

The holomorph of a group $G$ is the natural semidirect product 
\begin{equation*}
  \Aut(G) G
\end{equation*}
of $G$ by its automorphism group $\Aut(G)$. Let $S(G)$ be the group of
permutations on the set $G$. Consider the right and the left regular
representations of $G$:
\begin{equation*}
  \left\{
  \begin{aligned}
    \rho :\ &G \to S(G)
    \\&g \mapsto (x \mapsto x g) 
  \end{aligned}
  \right.
  \qquad
  \left\{
  \begin{aligned}
    \lambda :\ &G \to S(G)
    \\&g \mapsto (x \mapsto g x). 
  \end{aligned}
  \right.
\end{equation*}
\begin{notation}
  We denote the inversion map $g \mapsto g^{-1}$ on a group
  by $\inv$. 
\end{notation}
\begin{definition}
  The \emph{opposite} of the group $G$  is the group obtained by exchanging the
  order of factors in the product of $G$.
\end{definition}
The map $\inv$ is an isomorphism between a group $G$ and its opposite;
compare 
with Proposition~\ref{prop:right-and-left}\eqref{item:inversion} below.

The following well-known fact should be compared with Lemma~2.4.2
of the paper~\cite{GrePar} in which Greither and Pareigis set up the
connection, already mentioned in the Introduction, between Hopf Galois
extensions and regular subgroups of symmetric groups.
\begin{proposition}\label{prop:right-and-left}\

  \begin{enumerate}
  \item\label{item:right-and-left} $C_{S(G)}(\rho(G)) = \lambda(G)$
    and $C_{S(G)}(\lambda(G)) =  \rho(G)$.
  \item\label{lemma:AutG-is-stabilizer} The stabilizer of $1$ in
    $N_{S(G)}(\rho(G))$ is $\Aut(G)$. 
  \item\label{item:normalizer}  We have 
    \begin{equation*}
      N_{S(G)}(\rho(G))  =  \Aut(G)   \rho(G)  =  \Aut(G)  \lambda(G)  =
      N_{S(G)}(\lambda(G)),
    \end{equation*}
    and this group is  isomorphic to the holomorph  $\Aut(G) G$
    of $G$.
  \item\label{item:inversion}  Inversion on  $G$ normalizes
    $N_{S(G)}(\rho(G))$,
    centralizes $\Aut(G)$,  and conjugates $\rho(G)$  to $\lambda(G)$,
    that is
    \begin{equation*}
      \rho(G)^{\inv} = \lambda(G).
    \end{equation*}
  \end{enumerate}
\end{proposition}

\begin{notation}
  We write  $\Hol(G) = N_{S(G)}(\rho(G))$.

  We will  refer to either of
  the isomorphic groups
  $N_{S(G)}(\rho(G))$ and $\Aut(G) G$ as the holomorph of $G$.
\end{notation}



We now record another well-known fact.
\begin{lemma}\label{lemma:reghol}\

  \begin{enumerate}
  \item\label{item:reghol-1} Let $\Omega$ be a set, and $G$ a regular
    subgroup of 
    $S(\Omega)$. Then there is an isomorphism $S(\Omega) \to S(G)$ that
    sends $G$ to $\rho(G)$, and thus $N_{S(\Omega)}(G)$ to $\Hol(G)$.
  \item\label{item:reghol-2}   If $N \le S(G)$ is a regular subgroup, then
    $N_{S(G)}(N)$ is isomorphic to the holomorph of $N$.
  \end{enumerate}
\end{lemma}

It is because of Lemma~\ref{lemma:reghol}\eqref{item:reghol-1} that we
have done without a set $\Omega$, and started directly with $S(G)$ and
its regular subgroup $\rho(G)$.

\begin{proof}
  We only treat~\eqref{item:reghol-2}, for further reference.

  Consider for such a regular subgroup $N$ the bijection
  \begin{align*}
    \phi :\ &N \to G
    \\&n \mapsto 1^{n}.
  \end{align*}
  Then
  \begin{align*}
    \psi :\ &S(G) \to S(N)
          \\&\sigma \mapsto \phi \sigma \phi^{-1}
  \end{align*}
  (recall that we compose left-to-right) is an isomorphism, which maps $N$ onto
  $\rho(N)$, as for $x, n \in N$ we have
  \begin{equation*}
    x^{\phi n \phi^{-1}}
    =
    (1^{x n})^{\phi^{-1}}
    =
    x n,
  \end{equation*}
  that is, 
  \begin{equation*}
    \phi n \phi^{-1} = \rho(n).
  \end{equation*}
  In particular, $\psi(N_{S(G)}(N)) = N_{S(N)}(\rho(N)) = \Hol(N)$.
\end{proof}

\section{Groups with the same holomorph}
\label{sec:same-holomorph}

In  view  of  Lemma~\ref{lemma:reghol}\eqref{item:reghol-2},  one  may
inquire, what  are the regular subgroups  $N \le S(G)$ for which
\begin{equation}\label{eq:samehol}
  \Hol(N) \cong N_{S(G)}(N) = N_{S(G)}(\rho(G)) = \Hol(G).
\end{equation}

W.H.~Mills has noted in~\cite{Mills-multi} that if~\eqref{eq:samehol}
holds, then $G$ and $N$ need 
not be isomorphic. 
\begin{example}\label{ex:notiso}
  The dihedral and the generalized quaternion groups of order
  $2^{n}$, for $n \ge 4$, have the same normalizer in $S_{2^{n}}$, in
  suitable regular representations. \cite[3.10]{Kohl-twisted},
  \cite[2.1]{Kohl-multi}. 
\end{example}

When we restrict our attention to  the regular subgroups $N$ of $S(G)$
for which $N_{S(G)}(N) = \Hol(G)$ and $N  \cong G$, we can appeal to a
result   of    G.A.~Miller~\cite{Miller-multi}.    Miller    found   a
characterization  of these  subgroups in  terms of  the \emph{multiple
  holomorph} of $G$
\begin{equation*}
  \NHol(G) = N_{S(G)}(\Hol(G)).
\end{equation*}
Consider the set
\begin{equation*}
  \Hc(G)
  =
  \Set{ N \le S(G) : \text{$N$ is regular, $N \cong G$ and $N_{S(G)}(N)
    = \Hol(G)$} }.
\end{equation*}
Using the well-known fact that two regular subgroups of $S(G)$ are
isomorphic if and only if they are conjugate in $S(G)$, Miller showed
that the group $\NHol(G)$ acts transitively on $\Hc(G)$ by
conjugation. (See Lemma~\ref{lemma:conjugation} in the next Section
for a comment on this.) Clearly the 
stabilizer in $\NHol(G)$ of any element $N \in \Hc(G)$ is $N_{S(G)}(N)
= \Hol(G)$. We obtain
\begin{theorem}
  The group
  \begin{equation*}
    T(G) = \NHol(G) / \Hol(G)
  \end{equation*}
  acts regularly on $\Hc(G)$ by conjugation.
\end{theorem}

\section{Regular subgroups of the holomorph}
\label{sec:regular}

This section adapts to the nonabelian case the results
of~\cite[Theorem~1]{affine}~and \cite[Proposition 2]{FCC}.

Let  $G$   be  a  finite  group,   and  $N  \le  \Hol(G)$   a  regular
subgroup. Since $N$ is  regular, for each $g \in G$  there is a unique
element $\nu(g)  \in N$,  such that $1^{\nu(g)}  = g$.  Now $1^{\nu(g)
  \rho(g)^{-1}} =  1$, so  that $\nu(g)  \rho(g)^{-1} \in  \Aut(G)$ by
Proposition~\ref{prop:right-and-left}\eqref{lemma:AutG-is-stabilizer}. 
Therefore
for $g \in G$ we can write uniquely
\begin{equation}\label{eq:unique-form}
  \nu(g) = \gamma(g) \rho(g),
\end{equation}
for a suitable map $\gamma : G \to \Aut(G)$. We have
\begin{equation}\label{eq:nu}
  \nu(g) \nu(h)
  =
  \gamma(g) \rho(g) \gamma(h) \rho(h)
  =
  \gamma(g) \gamma(h) \rho(g^{\gamma(h)} h).
\end{equation}
Since $N$ is a subgroup of $S(G)$, $\gamma(g) \gamma(h) \in \Aut(G)$, and the
expression~\eqref{eq:unique-form} is unique, we have
\begin{equation*}
  \gamma(g) \gamma(h) \rho(g^{\gamma(h)} h) = \gamma( g^{\gamma(h)}
  h)\rho(g^{\gamma(h)} h),
\end{equation*}
from which we obtain
\begin{equation}\label{eq:gamma}
  \gamma(g) \gamma(h) = \gamma( g^{\gamma(h)} h ).
\end{equation}
It is now immediate to obtain
\begin{theorem}\label{thm:gamma-for-regular}
  Let $G$ be a finite group. The following data are equivalent.
  \begin{enumerate}
  \item A regular subgroup $N \le \Hol(G)$.
  \item A map $\gamma : G \to \Aut(G)$ such that
    \begin{equation}\label{eq:gamma-for-circ}
      \gamma(g) \gamma(h) = \gamma( g^{\gamma(h)} h ).
    \end{equation}
  \end{enumerate}
  Moreover, under these assumptions
  \begin{enumerate}[(a)]
  \item   the assignment
    \begin{equation*}
      g \circ h = g^{\gamma(h)} h.
    \end{equation*}
    for $g, h \in G$, defines a group structure $(G, \circ)$ with the
    same unity as that of $G$.
  \item There is an isomorphism $\nu : (G, \circ) \to N$.
  \item For $g, h \in G$, one has
    \begin{equation*}
      g^{\nu(h)} = g \circ h.
    \end{equation*}
  \end{enumerate}
\end{theorem}

\begin{proof}
  Concerning  the  last statements,  \eqref{eq:gamma-for-circ}  implies
  that $\circ$ is associative. Then for each $h \in G$ one has that $1
  \circ h = 1^{\gamma(h)} h = h$, as $\gamma(h) \in \Aut(G)$, and that
  $(h^{-1})^{\gamma(h)^{-1}}$ is a left inverse of $h$ with respect to
  $\circ$. The bijection $\nu$ introduced above is a homomorphism
  $(G, \circ) \to N$ by~\eqref{eq:nu}~and \eqref{eq:gamma}. Finally, 
  \begin{equation*}
    g^{\nu(h)} = g^{\gamma(h)} h = g \circ h.
  \end{equation*}
\end{proof}

Note, for later usage, that~\eqref{eq:gamma-for-circ} can be
rephrased, setting $k = g^{\gamma(h)}$, as
\begin{equation}\label{eq:gamma-for-dot}
  \gamma(k h) = \gamma(k^{\gamma(h)^{-1}}) \gamma(h).
\end{equation}

We record the following Lemma, which will be useful later. We use the
setup of~Theorem~\ref{thm:gamma-for-regular}.
\begin{lemma}\label{lemma:conjugation}
  Suppose $N \in \Hc(G)$, and let $\theta \in \NHol(G)$ such that
  $\rho(G)^{\theta} = N$ and $1^{\theta} = 1$. Then
  \begin{equation*}
    \theta : G \to (G, \circ)
  \end{equation*}
  is an isomorphism.

  Conversely, an isomorphism $\theta : G \to (G, \circ)$ conjugates
  $\rho(G)$ to $N$.
\end{lemma}

\begin{proof}
  Note first that given any $\theta \in \NHol(G)$ such that
  $\rho(G)^{\theta} = N$, we can modify $\theta$ by a suitable
  $\rho(g)$, and assume $1^{\theta} = 1$.

  Suppose for $y \in G$ one has $\rho(y)^{\theta} = \nu(y^{\sigma})$,
  for some $\sigma \in S(G)$. Thus $\rho(y) \theta = \theta
  \nu(y^{\sigma})$, so that for $x, y \in G$ one has 
  \begin{equation*}
    (x y)^{\theta} 
    =
    x^{\rho(y) \theta}
    =
    x^{\theta \nu(y^{\sigma})}
    =
    x^{\theta \gamma(y^{\sigma})} y^{\sigma}
    =
    x^{\theta} \circ y^{\sigma}.
  \end{equation*}
  Setting $x = 1$ we see that $\theta = \sigma$, and thus
  \begin{equation*}
    (x y)^{\theta} = x^{\theta} \circ y^{\theta}.
  \end{equation*}

  For the converse, if the last equation holds then
  \begin{equation*}
    x^{\rho(y)^{\theta}}
    =
    (x^{\theta^{-1}} y)^{\theta}
    =
    x \circ y^{\theta}
    =
    x^{\nu(y^{\theta})}.
  \end{equation*}
\end{proof}

\section{Normal regular subgroups of the holomorph}
\label{sec:normal-regular}

In this section, we adapt to the nonabelian case the results
of~\cite[Theorem 3.1]{fgag}.

Consider the sets
\begin{equation*}
  \Ic(G)
  =
  \Set{ N \le S(G) : \text{$N$ is regular, $N_{S(G)}(N) = \Hol(G)$} }
\end{equation*}
and
\begin{equation*}
  \Jc(G)
  =
  \Set{ N \le S(G) : \text{$N$ is regular, $N \norm \Hol(G)$} }.
\end{equation*}

Clearly 
we have
\begin{equation}\label{eq:HLK}
  \Hc(G) \subseteq \Ic(G) \subseteq \Jc(G).
\end{equation}
If  $N \norm \Hol(G)$, then $\Hol(G) \le N_{S(G)}(N)$. However the
latter may  well be properly  bigger than  the former, as shown  by the
following simple example.
\begin{example} 
  Let $G = \Span{  (1 \, 2\, 3 \, 4 ) }  \le S_{4}$. Then $N_{S_{4}}(G)$
  has order $8$, but its regular subgroup  $N = \Span{(1 \, 3) (2 \, 4),
    (1 \, 4) (2 \, 3)}$ is normal in the whole $S_{4}$.
\end{example}
Moreover, even  when $\Hol(G) =  N_{S(G)}(N)$, Example~\ref{ex:notiso}
shows that $G$ and $N$ are not necessarily isomorphic. Therefore all
inclusions in~\eqref{eq:HLK} may well be proper.

We will  now give a  characterization of  the elements of  $\Jc(G)$ in
terms  of  the   description  of  Theorem~\ref{thm:gamma-for-regular}.
Suppose $N \in \Jc(G)$. To ensure that $N \norm \Hol(G)$, it is enough to
make sure that $N$ is normalized by $\Aut(G)$, as if this holds,
then the normalizer of $N$ contains $\Aut(G) N$, which is contained in
$\Hol(G)$, and has the same order as $\Hol(G)$, as the regular
subgroup $N$ intersects $\Aut(G)$ trivially.

In order for $\Aut(G)$ to normalize $N$,
we  must  have  that for  all $\beta
\in  \Aut(G)$  and  $g  \in  G$,     the  conjugate
$\nu(g)^{\beta}$ of $\nu(g)$ by $\beta$ in $S(G)$ lies in $N$.  
Since
\begin{equation*}
  \nu(g)^{\beta}
  =
  (\gamma(g) \rho(g))^{\beta}
  =
  \gamma(g)^{\beta} \rho(g)^{\beta}
  =
  \gamma(g)^{\beta} \rho(g^{\beta})
\end{equation*}
and $\gamma(g)^{\beta} \in \Aut(G)$, uniqueness
of~\eqref{eq:unique-form} implies 
\begin{equation*}
  \gamma(g)^{\beta} \rho(g^{\beta})
  =
  \gamma(g^{\beta}) \rho(g^{\beta}),
\end{equation*}
so that
\begin{equation}\label{eq:gamma-is-autinv}
   \gamma(g^{\beta}) = \gamma(g)^{\beta}
\end{equation}
for $g \in G$ and $\beta \in \Aut(G)$. Applying this
to~\eqref{eq:gamma-for-dot}, we obtain that for $h, k \in G$
\begin{equation}\label{eq:gamma-is-anti}
  \gamma(k h)
  =
  \gamma(k^{\gamma(h)^{-1}}) \gamma(h)
  =
  \gamma(k)^{\gamma(h)^{-1}} \gamma(h)
  =
  \gamma(h) \gamma(k),
\end{equation}
that is, $\gamma : G \to \Aut(G)$ is an antihomomorphism.

Now note that~\eqref{eq:gamma-for-circ}  follows
from~\eqref{eq:gamma-is-autinv}~and \eqref{eq:gamma-is-anti}, 
as
\begin{equation*}
  \gamma( g^{\gamma(h)} h )
  =
  \gamma(h) \gamma(g)^{\gamma(h)}
  =
  \gamma(g) \gamma(h).
\end{equation*}
We have obtained
\begin{theorem}\label{thm:normal-regular}
Let $G$ be a finite group. The following data are equivalent.
  \begin{enumerate}
  \item A regular subgroup $N \norm \Hol(G)$, that is, an element of $\Jc(G)$.
  \item A map $\gamma : G \to \Aut(G)$ such that for $g, h \in G$ and
    $\beta \in \Aut(G)$
    \begin{equation}\label{eq:gamma-for-normal}
      \begin{cases}
        \gamma(g h) = \gamma(h) \gamma(g)\\
        \gamma(g^{\beta}) = \gamma(g)^{\beta}.
      \end{cases}
    \end{equation}
  \end{enumerate}
  Moreover, under these assumptions
  \begin{enumerate}[(a)]
  \item   the assignment
    \begin{equation*}
      g \circ h = g^{\gamma(h)} h.
    \end{equation*}
    for $g, h \in G$, defines a group structure $(G, \circ)$ with the
    same unity as that of $G$.
  \item There is an isomorphism $\nu : (G, \circ) \to N$.
  \item For $g, h \in G$, one has
    \begin{equation*}
      g^{\nu(h)} = g \circ h.
    \end{equation*}
  \item\label{item:every-auto-tells-a-story}  Every automorphism  of
    $G$  is also  an automorphism  of $(G, 
    \circ)$. 
  \end{enumerate}
\end{theorem}

\begin{remark}\label{rem:braces}
  Note that under the hypotheses of Theorem~\ref{thm:gamma-for-regular},
  $G$ becomes a \emph{skew right brace} (for which see~\cite{bra})
  under the operations $\cdot$ and 
  $\circ$, that is, $G$ is a group with respect to both operations,
  which are connected by
  \begin{equation*}
    (g h) \circ k = (g \circ k) k^{-1} (h \circ k),
    \qquad
    \text{for $g, h, k \in G$}.
  \end{equation*}
  The braces which correspond to the normal regular subgroups of $\Hol(G)$
  satisfy the additional condition of
  Theorem~\ref{thm:normal-regular}\eqref{item:every-auto-tells-a-story}.   
\end{remark}

In the following, when dealing with $N \in \Jc(G)$, we will be using
the notation of Theorem~\ref{thm:normal-regular} without further
mention. 

\begin{proof}
The last statement follows from
\begin{equation*}
  (g \circ h)^{\beta}
  =
  (g^{\gamma(h)} h)^{\beta}
  =
  (g^{\beta})^{\gamma(h)^{\beta}} h^{\beta}
  =
  (g^{\beta})^{\gamma(h^{\beta})} h^{\beta}
  =
  g^{\beta} \circ  h^{\beta},
\end{equation*}
for $g, h \in G$ and $\beta \in \Aut(G)$.
\end{proof}

Let us exemplify the above for the case of the left regular
representation. Consider the morphism
\begin{align*}
  \iota :\ &G \to \Aut(G)
         \\&y \mapsto (x \mapsto y^{-1} x y),
\end{align*}
that is, $\iota(y) \in \Inn(G)$ is conjugacy by $y$. If $N =
\lambda(G)$, then we have for $y \in G$ 
\begin{equation*}
  \lambda(y) = \iota(y^{-1}) \rho(y),
\end{equation*}
as for $z, y \in G$ we have
\begin{equation*}
  z^{\iota(y^{-1}) \rho(y)}
  =
  y z y^{-1} y
  =
  y z
  =
  z^{\lambda(y)}.
\end{equation*}
Therefore $\gamma(y) = \iota(y^{-1})$, and
\begin{equation*}
  x \circ y = x^{\iota(y^{-1})} y = y x,
\end{equation*}
that is, $(G, \circ)$ is the opposite group of $G$.

Also, in~\cite{CarChi} S.~Carnahan and L.~Childs  prove that if $G$ is
a  non-abelian finite  simple  group, then  $\Hc(G)  = \Set{  \rho(G),
  \lambda(G) }$.   In our context, this  can be proved as  follows. If
$G$ is a non-abelian finite simple group, and $N \in \Hc(G)$, then the
normal  subgroup $\ker(\gamma)$  of  $G$  can only  be  either $G$  or
$\Set{1}$. In the first case we have  $x \circ y = x^{\gamma(y)} y = x
y$ for $x, y \in G$ , so that
\begin{equation*}
  x^{\nu(y)} = x \circ y = x y = x^{\rho(y)},
\end{equation*}
and $N = \rho(G)$. In the second case, $\gamma$ is injective. Since we
have
\begin{equation}\label{eq:circinv}
  \gamma(x \circ y) = \gamma(x) \gamma(y) = \gamma(y x),
\end{equation}
we obtain $x \circ y = y x$, so $N = \lambda(G)$ as we have just seen.

\section{Commutators}
\label{sec:commutators}

In this section we assume we are in the situation of
Theorem~\ref{thm:normal-regular}. 

Let $\beta  \in \Aut(G)$,  $g  \in G$, and consider  the commutator
$[\beta,   g^{-1}]    =   g^{\beta}   g^{-1}$   taken    in   $\Aut(G)
G$. Using~\eqref{eq:gamma-for-normal}, we get
\begin{equation}\label{eq:self-adjoint-on-autos}
  \gamma([\beta, g^{-1}])
  =
  \gamma(g^{\beta} g^{-1})
  =
  \gamma(g)^{-1} \gamma(g)^{\beta}
  =
  [\gamma(g), \beta].
\end{equation}
In the particular case when $\beta = \iota(h)$, for some $h \in G$, we
obtain
\begin{equation*}
   \gamma([h, g^{-1}])
   =
   \gamma([\iota(h), g^{-1}])
   =
   [\gamma(g), \iota(h)]
   =
   \iota([\gamma(g), h]),
\end{equation*}
that is
\begin{equation}\label{eq:gamma-of-commutators}
  \gamma([h, g^{-1}]) = \iota([\gamma(g), h]).
\end{equation}
From
this identity we obtain
\begin{multline*}
  \iota([\gamma(g), h])
  =
  \gamma([h, g^{-1}])
  =
  \gamma([g^{-1}, h])^{-1}
  =\\=
  \iota([\gamma(h^{-1}), g^{-1}])^{-1} 
  =
  \iota([g^{-1}, \gamma(h^{-1})]),
\end{multline*}
that is,
\begin{equation}\label{eq:self-adjoint}
  [\gamma(g), h] \equiv [g^{-1}, \gamma(h^{-1})] \pmod{Z(G)}
\end{equation}
for all $g, h \in G$.

In the rest of the paper we will deal  with the case of finite perfect groups,
that is,  those finite  groups $G$  such that $G'=  G$. In  this case,
according to~\eqref{eq:gamma-of-commutators},  we have  $\gamma(G) \le
\Inn(G)$.

\section{Perfect groups}
\label{sec:main}

Let $G$  be a  non-trivial, finite, perfect  group. We  will determine
$\Jc(G)$, and then discuss its relationship to $\Hc(G)$.

Recall that an
automorphism $\beta$ of a group $G$ is said to be \emph{central} if
$[x, \beta] = x^{-1} x^{\beta} \in Z(G)$ for all $x \in G$. In other
words, an automorphism of $G$ is central if it induces the identity on
$G / Z(G)$.

We  record for later usage a couple of elementary, well-known facts. 
\begin{lemma}\label{lemma:perfect}
  Let $G$ be a finite perfect group.
  \begin{enumerate}
  \item $Z_{2}(G) = Z(G)$.
  \item\label{item:central-trivial} A central automorphism of $G$ is trivial.
  \end{enumerate}
\end{lemma}

\begin{proof}
  The first part is Gr\"un's Lemma~\cite{gruen}.

  For the second part, if $\beta$ is a central automorphism of $G$,
  then 
  \begin{equation*}
    x \mapsto [x, \beta]
  \end{equation*}
  is a homomorphism from $G$ to $Z(G)$. Since $G = G'$, this
  homomorphism maps $G$ onto the identity.  
\end{proof}

We now  show that an  element $N \in  \Jc(G)$ yields a  direct product
decomposition of $\Inn(G)$.
\begin{proposition}\label{prop:perfect}
  Let $G$ be a finite, perfect group, and $N \in \Jc(G)$.
  \begin{enumerate}
  \item\label{item:perfect-center} $Z(G) \le \ker(\gamma)$.
  \item\label{item:direct-product} $\Inn(G) = \gamma(G) \times
    \iota(\ker(\gamma))$. 
  \end{enumerate}
\end{proposition}

Later we will lift the direct product
decomposition~\eqref{item:direct-product} of $\Inn(G)$ to a central
product decomposition of $G$ (Theorem~\ref{thm:cpd}\eqref{item:cpd}).

\begin{proof}
  For the first part, let $g \in Z(G)$. \eqref{eq:self-adjoint} yields
  $[\gamma(g), h] \in Z(G)$ for all $h \in G$, that is, $\gamma(g)$ is
  a central automorphism of $G$. By
  Lemma~\ref{lemma:perfect}\eqref{item:central-trivial}, 
  $\gamma(g) = 1$.

  For  the   second  part,   we  first   show  that   $\gamma(G)$  and
  $\iota(\ker(\gamma))$ commute  elementwise. Let $g \in  G$ and $k
  \in \ker(\gamma)$.   The results  of Section~\ref{sec:commutators}
  yield
  \begin{equation*}
    [\gamma(g), \iota(k)]
    =
    \iota([\gamma(g), k])
    =
    \iota([g^{-1}, \gamma(k^{-1})])
    =
    1.
  \end{equation*}
  
  We now show that $\gamma(G) \cap \iota(\ker(\gamma)) = 1$. Write an
  element of the perfect group $G$ as
  \begin{equation*}
    x = \prod_{i=1}^{n} [h_{i}, g_{i}^{-1}],
  \end{equation*}
  for suitable $g_{i}, h_{i} \in G$.
  
  Using the first identity of~\eqref{eq:gamma-for-normal} we
  get first
  \begin{equation}\label{eq:gamma-of-x}
    \gamma(x)
    =
    \prod_{i=n}^{1} \gamma([h_{i}, g_{i}^{-1}])
    =
    \prod_{i=n}^{1} [\gamma(g_{i}), \gamma(h_{i}^{-1})]
  \end{equation}
  (note that the order of the product has been inverted by the
  application of $\gamma$).
  
  Using~\eqref{eq:gamma-of-commutators}   and   the   first   identity
  of~\eqref{eq:gamma-for-normal} we also get
  \begin{align*}
    \gamma(x)
    =
    \prod_{i=n}^{1} \gamma([h_{i}, g_{i}^{-1}])
    =
    \prod_{i=n}^{1} \iota([\gamma(g_{i}), h_{i}])
    =
    \iota\left(\prod_{i=n}^{1} [\gamma(g_{i}), h_{i}]\right).
  \end{align*}
  
  Now if $\gamma(x) \in \gamma(G) \cap \iota(\ker(\gamma))$,
  part~\eqref{item:perfect-center} yields 
  \begin{equation*}
    \prod_{i=n}^{1} [\gamma(g_{i}), h_{i}] \in \ker(\gamma).
  \end{equation*}
  We thus have, using~\eqref{eq:gamma-for-normal}~and
  \eqref{eq:self-adjoint-on-autos} 
  \begin{align*}
    1
    =
    \gamma \left(\prod_{i=n}^{1} [\gamma(g_{i}), h_{i}]\right)
    =
    \prod_{i=1}^{n}
    \gamma([\gamma(g_{i}), h_{i}])
    =
    \prod_{i=1}^{n}
    [\gamma(h_{i}^{-1}), \gamma(g_{i})]
    =
    \gamma(x)^{-1},
  \end{align*}
  according    to~\eqref{eq:gamma-of-x}.
  Therefore $\gamma(x) = 1$, as claimed.

  Finally we have, keeping in mind part~\eqref{item:perfect-center},
  \begin{equation*}
    \Size{\gamma(G) \times \iota(\ker(\gamma))}
    =
    \Size{\gamma(G)} \cdot \frac{\Size{\ker(\gamma)}}{\Size{Z(G)}}
    =
    \Size{G / Z(G)}
    =
    \Size{\Inn(G)},
  \end{equation*}
  so that $\gamma(G) \times \iota(\ker(\gamma)) = \Inn(G)$.
\end{proof}

Regarding $\Inn(G)$ and  $G$ as groups with  operator group $\Aut(G)$,
we  note  that   the  second  equation  of~\eqref{eq:gamma-for-normal}
implies   that   both   $\gamma(G)$  and   $\iota(\ker(\gamma))$   are
$\Aut(G)$-invariant,  and  so  are  $H  =  \iota^{-1}(\gamma(G))$  and
$\ker(\gamma)$.  (Clearly the  latter statement is the  same as saying
that $H$ and  $\ker(\gamma)$ are characteristic subgroups  of $G$, but
we prefer  to use the same  terminology of groups with  $\Aut(G)$ as a
group  of operators  for  both $G$  and  $\Inn(G)$.) 

We  have  $G =  H
\ker(\gamma)$.
We claim  that $[H, \ker(\gamma)]  = 1$, that  is, $G$ is  the central
product of $H$  and $\ker(\gamma)$, amalgamating $Z(G)$.  We will need
the following simple Lemma, which is hinted at
by Joshua A.\ Grochow and Youming Qiao in~\cite[Remark 7.6]{GroQiao}.
\begin{lemma}\label{lemma:perfect-lifting}
Let $G$ be a group, and $H, K \le G$ such that
 \begin{equation*}
G/Z(G) = H Z(G) / Z(G) \times K Z(G) / Z(G).
\end{equation*}
Suppose $K Z(G) / Z(G)$ is perfect.

Then 
\begin{enumerate}
\item $K'$ is perfect, and
\item\label{item:commute} $[H, K] = 1$.
\end{enumerate}
\end{lemma}

\begin{proof}
Since  $K Z(G) / Z(G)$ is
perfect, we have $K Z(G) = K' Z(G)$, so that $K' = K''$, and $K'$ is
perfect. As $[H, K'] = [H, K' Z(G)]  = [H, K Z(G)] = [H, K] \le Z(G)$,
$H$ induces by  conjugation central automorphisms on $K'$,  so that by
Lemma~\ref{lemma:perfect}\eqref{item:central-trivial}  $[H,  K] =  [H,
  K'] = 1$.
\end{proof}

In our situation, take $K =
\ker(\gamma)$. We have that $K Z(G) / Z(G) \cong \iota(K)$ is perfect,
as a direct factor of the perfect group $G/Z(G)$.
Then  Lemma~\ref{lemma:perfect-lifting}\eqref{item:commute} implies  that $G$
is the central product of $H$ and $\ker(\gamma)$, amalgamating $Z(G)$.

We claim  
\begin{lemma}
$\gamma(y) = \iota(y^{-1})$  for  $y \in  H$. 
\end{lemma}

\begin{proof}
  Let  $y \in  H$. We  claim that  $\gamma(y) \iota(y)$  is a  central
  automorphism    of    the    perfect    group    $G$,    so that
  by~\ref{lemma:perfect}\eqref{item:central-trivial}    $\gamma(y)   =
  \iota(y^{-1})$.

  If $x \in K = \ker(\gamma)$, we have
  from~\eqref{eq:self-adjoint}
  \begin{equation*}
    [\gamma(y), x] \equiv [y^{-1}, \gamma(x^{-1})] \equiv 1 \pmod{Z(G)},
  \end{equation*}
  so that $\gamma(y)$ induces an automorphism of the characteristic
  subgroup $K$ which is the identity
  modulo $Z(G)$, and so does $\gamma(y)
  \iota(y)$, as $[H, K] = 1$.

  Let now $x \in H$. Consider first the special case when $G = H
  \times K$. Then $\gamma$ is injective on $H$, so
  that~\eqref{eq:circinv} implies $x \circ y = y x$, and thus
  \begin{equation*}
    x^{-1} x^{\gamma(y) \iota(y)} 
    = 
    x^{-1} y^{-1} (x \circ y) y^{-1} y 
    = 
    x^{-1} y^{-1} (y x) y^{-1} y 
    = 
    1,
  \end{equation*}
  that is,
  $\gamma(y) \iota(y)$ is the identity on $H$.

  In the general case, \eqref{eq:circinv} implies that $x \circ y
  \equiv y x \pmod{K}$, so that as above
  \begin{equation*}
    x^{-1}  x^{\gamma(y) \iota(y)}  \equiv x^{-1}  y^{-1} (x  \circ y)
    y^{-1} y \equiv 1 \pmod{K},
  \end{equation*}
  that is, $x^{-1} x^{\gamma(y) \iota(y)} \in K$.
  Clearly $x^{-1} x^{\gamma(y) \iota(y)} \in H$, as $H$ is
  characteristic in $G$. Therefore $x^{-1} x^{\gamma(y) \iota(y)} \in
  H \cap K \le Z(G)$, so that $\gamma(y) \iota(y)$ induces an
  automorphism of $H$ which is the identity modulo $Z(G)$.

  It follows  that $\gamma(y) \iota(y)$  is a central  automorphism of
  $G = H K$, as claimed.
\end{proof}

For $y \in H$ and $x \in \ker(\gamma)$ we have
\begin{equation*}
  x \circ y 
  = 
  x^{\gamma(y)} y 
  =
  x^{\iota(y^{-1})} y
  =
  y x y^{-1} y 
  = 
  y x 
  = 
  y^{\gamma(x)} x
  =
  y \circ x.
\end{equation*}
Also,  if $x, y \in H$ we have $x \circ y = x^{\gamma(y)} y = x^{y^{-1}}
y = y x$.

In the following we will be  writing the elements of $G$ as pairs
in $H  \times \ker(\gamma)$, understanding  that a pair  represents an
equivalence  class with  respect  to the  central product  equivalence
relation which identifies  $(x z,  y)$ with $(x, z  y)$, for $z
\in Z(G)$. 

We have obtained
\begin{theorem}\label{thm:cpd}
  Let $G$ be a finite perfect group. 
  \begin{enumerate}
  \item\label{item:cpd}
    If $N \in \Jc(G)$, then $G$ is  a central product of its subgroups $H
    =   \iota^{-1}(\gamma(G))'$   and   $K = \ker(\gamma)$.  Both   $H$   and
    $K$ are $\Aut(G)$-subgroups of $G$.
  \item For $x \in H$ we have $\gamma(x) = \iota(x^{-1})$.
  \item\label{item:circ}
    $(G, \circ)$ is  also a central product of the  same subgroups. If
    we represent the elements of $G$ as (equivalence classes of) pairs
    in $H \times K$, then
    \begin{equation}\label{eq:product-on-N}
      (x, y) \circ (a, b) = (a x, y b).
    \end{equation}
  \item
    For $(a, b) \in G$, the action of $\nu(a, b)$ on $(x, y) \in G$ is given by
    \begin{equation*}
      (x, y)^{\nu(a, b)}
      =
      (x, y) \circ (a, b)
      =
      (a x, y b),
    \end{equation*}
    that is, $N$ induces the right regular representation on
    $K$, and the left regular representation on $H$.
  \end{enumerate}
\end{theorem}

We note the following analogue of
Proposition~\ref{prop:right-and-left}\eqref{item:right-and-left}~and
\eqref{item:inversion}. 
\begin{proposition}
Let $G$ be a finite perfect group, and let $G = H K$ be a central
decomposition, with $\Aut(G)$-invariant subgroups $H, K$. 
Consider the following two elements of $\Jc(G)$.
\begin{enumerate}
\item  $N_{1}$,  for  which  $\ker(\gamma_{1})  = K  Z(G)$  and  $H  =
  \iota^{-1}(\gamma_{1}(G))'$,  with  $\gamma_{1}(x) =  \iota(x^{-1})$
  for $x \in H$, and associated  group operation $(x, y) \circ_{1} (a,
  b) = (a x, y b)$.
\item  $N_{2}$,  for  which  $\ker(\gamma_{2})  = H  Z(G)$  and  $K  =
  \iota^{-1}(\gamma_{2}(G))'$,  with  $\gamma_{2}(x) =  \iota(x^{-1})$
  for $x \in K$, and associated  group operation $(x, y) \circ_{2} (a,
  b) = (x a, b y)$.
\end{enumerate}
Then
\begin{enumerate}
\item $N_{1}^{\inv} = N_{2}$. 
\item $\inv: (N_{1}, \circ_{1}) \to (N_{2}, \circ_{2})$ is an isomorphism.
\item $C_{S(G)}(N_{1}) = N_{2}$ and $C_{S(G)}(N_{2}) = N_{1}$.
\end{enumerate}
\end{proposition}

\begin{proof}
The proof is straightforward. If $N_{i} = \Set{  \nu_{i}(a, b) : (a,
  b) \in G}$ as in Section~\ref{sec:regular},we have 
\begin{multline*}
  (x, y)^{\inv \nu_{1}(a, b) \inv}
  =
  (x^{-1}, y^{-1})^{\nu_{1}(a, b) \inv}
  =\\=
  (a x^{-1}, y^{-1} b)^{\inv}
  =
  (x a^{-1}, b^{-1} y)
  =
 (x, y)^{\nu_{2}((a, b)^{\inv})},
\end{multline*}
and then, as in Lemma~\ref{lemma:conjugation}, $\inv: (N_{1},
\circ_{1}) \to (N_{2}, \circ_{2})$ is an isomorphism. 
\end{proof}

We now give a description of all possible central product
decompositions of the perfect group $G$ as in Theorem~\ref{thm:cpd}.

We deal first  with the particular case when $Z(G) =  1$, where we are
able to show  that $\Jc(G) = \Ic(G) = \Hc(G)$  and determine this set,
and the group $T(G) = \NHol(G)  / \Hol(G)$.  When $Z(G)$ is allowed to
be non-trivial, we are able  to determine $\Jc(G)$. However, 
examples show that in this case $\Hc(G)$, $\Ic(G)$ and
$\Jc(G)$ can be distinct, and we 
are unable at the moment to describe $T(G)$.

\subsection{The centerless case}
\label{sub:centerless}

Suppose $Z(G) = 1$, so that $\iota : G \to \Inn(G)$ is an isomorphism
of $\Aut(G)$-groups. 

Consider a Krull-Remak-Schmidt decomposition
\begin{equation*}
G = A_{1} \times A_{2} \times \dots \times A_{n}
\end{equation*}
of  $G$   as  an   $\Aut(G)$-group.   Since  $Z(G)   =  1$,   this  is
unique~\cite[3.3.8,  p.   83]{Rob1996}.   Therefore  the only  way  to
decompose  $G$ as  the ordered  direct product  of  two characteristic
subgroups $H,  K$ is by grouping  together the $A_{i}$,  so that there
are $2^{n}$ ways  of doing this. If $G  = H \times K$ is  one of these
ordered  decompositions, define  an antihomomorphism  $\gamma :  G \to
\Aut(G)$  by  $\gamma(k)  =  1$  for  $k  \in  K$,  and  $\gamma(h)  =
\iota(h^{-1})$, for $h \in H$. Then $\gamma$ satisfies also the second
identity  of~\eqref{eq:gamma-for-normal},  and  we  have  obtained  an
element          $N         \in         \Jc(G)$          as         in
Theorem~\eqref{thm:cpd}\eqref{item:circ}.  The involution  $\theta \in
\NHol(G)$ given by $(h, k)^ {\theta} = (h^{-1}, k)$, for $h \in H$ and
$k \in K$ is an isomorphism $G \to (G, \circ)$. We have obtained
\begin{theorem}\label{thm:main-for-centerless}
  Let $G$ be a finite perfect group with $Z(G) = 1$. 
  \begin{enumerate}
  \item
    If $N  \in \Jc(G)$, that is,  $N \norm \Hol(G)$ is  regular, then $N
    \in \Hc(G)$, that is, $N \cong G$.
  \item If $n$  is the length of a  Krull-Remak-Schmidt decomposition of
    $G$ as an $\Aut(G)$-group, then $\Hc(G)$ has $2^{n}$ elements.
  \item $T(G)$ is an elementary abelian group of order $2^{n}$.
  \end{enumerate}
\end{theorem}

\subsection{Non-trivial center}
\label{sub:centermore}

We  now  consider  the  situation  when  $Z(G)$  is  (allowed  to  be)
non-trivial.

We describe the  elements $N$ of $\Jc(G)$,  in analogy  with the
centerless case. 

As in  the centerless  case, we  may consider  the Krull-Remak-Schmidt
decomposition
\begin{equation}\label{eq:KRS-for-Inn}
\Inn(G) = A_{1} \times A_{2} \times \dots \times A_{n}
\end{equation}
of $\Inn(G)$ as  an $\Aut(G)$-group. This corresponds  uniquely to the
central product decomposition of $G$
\begin{equation}\label{eq:cpd}
G = B_{1}  B_{2}  \cdots  B_{n},
\end{equation}
where    $B_{i}     =    \iota^{-1}(A_{i})'$     are    \emph{perfect}
$\Aut(G)$-subgroups,   which    are   centrally    indecomposable   as
$\Aut(G)$-subgroups.       Therefore      the     central      product
decomposition~\eqref{eq:cpd}  is also  unique. (Recall  also that  the
Krull-Remak-Schmidt    of    $G$    in   terms    of    indecomposable
$\Aut(G)$-subgroups    is   unique,    because   of~\cite[3.3.8,    p.
  83]{Rob1996}~and
Lemma~\ref{lemma:perfect}.\eqref{item:central-trivial}.)

As in the  centerless case, we obtain that every  decomposition $G = H
\ker(\gamma)$ as in Theorem~\ref{thm:cpd}  can be obtained by grouping
together the $B_{i}$  in two subgroups $H$ and $K$,  and then defining
an antihomomorphism $\gamma : G \to \Aut(G)$ by $\gamma(k) = 1$ for $k
\in K  Z(G)$, and $\gamma(h) =  \iota(h^{-1})$, for $h \in  H$, and
then $\circ$ as in~\eqref{eq:product-on-N}.  As in
the centerless case,  this yields an element $N  \in \Jc(G)$. Moreover
$(G, \circ)$  is still  a central  product of $H$  and $K  Z(G)$, with
$\circ$ as in Theorem~\ref{thm:cpd}\eqref{item:circ}.

We have obtained the following weaker analogue of
Theorem~\ref{thm:main-for-centerless}. 
\begin{theorem}\label{thm:main-for-centermore}
  Let $G$ be a finite perfect group.

  If $n$  is the length of a  Krull-Remak-Schmidt decomposition of
  $\Inn(G)$ as an $\Aut(G)$-group, then $\Jc(G)$ has $2^{n}$ elements,
  that is, there are $2^{n}$ regular subgroups $N \norm \Hol(G)$.
\end{theorem}

Write  $Y =  H  \cap KZ(G)$.  As  the subgroups  $H$  and $KZ(G)$  are
characteristic in $G$, we obtain that  the elements of $\Aut(G)$ can
be described via the set of pairs
\begin{equation*}
  \Set{ (\sigma, \tau) : \sigma \in \Aut(H), \tau \in \Aut(K Z(G)),
  \sigma\restriction_{Y} = \tau\restriction_{Y} }.
\end{equation*}
Theorem~\ref{thm:normal-regular}\eqref{item:every-auto-tells-a-story}
states that $\Aut(G) \le \Aut(G, \circ)$. 
However, the latter group might well be bigger than the former. This
is shown by the following
\begin{proposition}\label{prop:not-char}
  There exist perfect and centrally indecomposable groups $Q_{1}, Q_{2},
  Q_{3}$, and a central product $G = Q_{1} Q_{2} Q_{3}$, such that
  \begin{enumerate}
  \item each $Q_{i}$ is characteristic in $G$,
  \item in the group $(G, \circ)$ obtained by replacing $Q_{1}$ with
    its opposite, the subgroups $Q_{1}$ and $Q_{2}$ are exchanged by an
    automorphism of $(G, \circ)$, and thus are not
    characteristic in $(G, \circ)$.
  \end{enumerate}
\end{proposition}
Clearly the automorphism of the second condition lies in $\Aut(G,
\circ) \setminus \Aut(G)$. This example shows that $\Ic(G)$ may well
be a proper subset of $\Jc(G)$. 

Clearly $N \in \Ic(G)$ if and only if $\Aut(G) = \Aut(G, \circ)$.
However, in this general situation $\Hc(G)$ might be a proper subset
of $\Ic(G)$. This is shown by the following
\begin{proposition}\label{prop:non-iso}
  There  exist  perfect and centrally  indecomposable  groups  $Q_{1},
  Q_{2}$, and  a central product $G  = Q_{1} Q_{2}$ such that
  \begin{enumerate}
  \item $Q_{1}, Q_{2}$ are characteristic in $G$;
  \item the group $(G, \circ)$  obtained from $G$ by replacing $Q_{1}$
    with its opposite $G$ is not isomorphic to $G$.
  \end{enumerate}
  Moreover, $Q_{1}$ and $Q_{2}$ are still characteristic in $(G, \circ)$.
\end{proposition}
Thus if $N$ is the regular subgroup corresponding to $(G, \circ)$ of this
Proposition, we have $N \in \Ic(G) \setminus \Hc(G)$.

So on  the one  hand not  all central  product decompositions  lead to
regular subgroups $N$  which are isomorphic to $G$. And  even when $N$
is isomorphic  to $G$, it is  not clear whether $\rho(G)$  and $N$ are
conjugate under an involution in $T(G)$, and therefore it is not clear
to  us at  the moment  whether $T(G)$  is elementary  abelian in  this
general case.

To  construct  the   groups  of  Proposition~\ref{prop:not-char}~and
\ref{prop:non-iso}, we  rely on the following family  of examples, based
on a construction that we have learned from Derek Holt.
\begin{proposition}\label{lemma:After-Holt}
  There exists a family of groups $L_{p}$, for  $p \equiv  1 \pmod{3}$  a prime,
  with the following properties:
  \begin{enumerate}
  \item the groups $L_{p}/Z(L_{p})$ are pairwise non-isomorphic, 
  \item the groups $L_{p}$ are perfect, and centrally
  indecomposable,
  \item $Z(L_{p})$ is of order $3$, and
  \item $\Aut(L_{p})$ acts trivially on $Z(L_{p})$.
  \end{enumerate}
\end{proposition}
This is proved in Section~\ref{sec:After-Holt}.

\begin{proof}[Proof of Proposition~\ref{prop:not-char}]
  Let $Q_{1}, Q_{2}$ be two isomorphic copies of one of the groups of
  Proposition~\ref{lemma:After-Holt}, and $Q_{3}$ another group as in
  Proposition~\ref{lemma:After-Holt}, not isomorphic to $Q_{1},
  Q_{2}$.

  Fix an isomorphism $\zeta: Q_{1} \to Q_{2}$. If $Z(Q_{1}) =
  \Span{a_{1}}$, let $a_{2} = a_{1}^{\zeta}$. Let $Z(Q_{3}) =
  \Span{b}$.

  Consider the central product $G = Q_{1} Q_{2} Q_{3}$, amalgamating
  $a_{2} = a_{1}^{-1} = b$.

  Consider   the    quotient   $G/Z(G)$.    Since   the    groups   of
  Proposition~\ref{lemma:After-Holt} are centrally indecomposable, and
  have  pairwise   non-isomorphic  central   quotients,  Krull-Remak-Schmidt
  implies that $Q_{3}/Z(G)$  is characteristic in $G/Z(G)$,  and so is
  $Q_{1}  Q_{2}  / Z(G)$.  Therefore  $Q_{3}$  and $Q_{1}  Q_{2}$  are
  characteristic  in  $G$.   Moreover,  if there  is  an  automorphism
  $\alpha$ of $G$  that does not map $Q_{1}$ to  itself, then applying
  to  $Q_{1} Q_{2}  / Z(G)$  either Krull-Remak-Schmidt,  or the  results of
  \cite[Theorem~3.1]{Bid},   and   using    either   the   fact   that
  $Q_{1}/Z(Q_{1})$ and $Q_{2}/Z(Q_{2})$ are  perfect, or that they are
  centerless, we obtain that $\alpha$ exchanges $Q_{1}$ and $Q_{2}$.

  Therefore $\alpha\restriction_{Q_{1}}  \zeta^{-1}$  is an
  automorphism  of   $Q_{1}$,  and  thus  maps   $a_{1}$  to  $a_{1}$.
  Therefore  $\alpha$  maps  $a_{1}$   to  $a_{1}^{\zeta}  =  a_{2}  =
  a_{1}^{-1}$.  But $\alpha\restriction_{Q_{3}}$ is an automorphism of
  $Q_{3}$, and thus fixes $b = a_{1}$, a contradiction.

  Consider now the group $(G, \circ)$ obtained by replacing $Q_{1}$
  with its opposite. Now the map which is the identity on $Q_{3}$,
  $\zeta \inv$ on $Q_{1}$ and $\zeta^{-1} \inv$ on $Q_{2}$ induces an
  automorphism of $(G, \circ)$ which exchanges $Q_{1}$ and $Q_{2}$. In
  fact we have for $x, y \in Q_{1}$
  \begin{equation*}
    (x \circ y)^{\zeta \inv}
    =
    (y x)^{\zeta \inv}
    =
    (y^{\zeta} x^{\zeta})^{\inv}
    =
    x^{\zeta \inv} y^{\zeta \inv}, 
  \end{equation*}
  and for $x, y \in Q_{2}$
  \begin{multline*}
    (x \circ y)^{\zeta^{-1} \inv}
    =
    (x y)^{\zeta^{-1} \inv}
    =
    (x^{\zeta^{-1}} y^{\zeta^{-1}})^{\inv}
    =\\=
    y^{\zeta^{-1} \inv} x^{\zeta^{-1} \inv}
    =
    x^{\zeta^{-1} \inv} \circ y^{\zeta^{-1} \inv}.
  \end{multline*}
  Moreover $a_{1}^{\zeta  \inv} = a_{2}^{\inv} =  a_{2}^{-1} = a_{1}$,
  and $a_{2}^{\zeta^{-1}  \inv} = a_{1}^{\inv} =  a_{1}^{-1} = a_{2}$,
  which is compatible with the identity on $Q_{3}$.
\end{proof}

\begin{proof}[Proof of Proposition~\ref{prop:non-iso}]
  Let $Q_{1}, Q_{2}$ to be two non-isomorphic groups as in
  Proposition~\ref{lemma:After-Holt}, and let $G = Q_{1} Q_{2}$,
  amalgamating the centers. Consider the group $(G, \circ)$ obtained
  by replacing $Q_{1}$ with its opposite. If the map $\theta : G \to G$
  is an isomorphism of $G$ onto $(G, \circ)$, by the arguments of the
  previous proofs it has to map
  each $Q_{i}$ to 
  itself. Then $\theta$ induces an anti-automorphism on $Q_{1}$, thus
  inverting $Z(G)$, and an automorphism on $Q_{2}$, thus fixing $Z(G)$
  elementwise, a contradiction.
\end{proof}

\section{Proof of Proposition~\ref{lemma:After-Holt}}
\label{sec:After-Holt}

Consider the groups $T_{p} = \PSL(3,  p)$, where $p \equiv 1 \pmod{3}$
is a prime, and let $\F_{p}$ be the field with $p$ elements.

It is well known (\cite[Theorem 3.2]{simple}) that
the outer automorphism group of $T_{p}$ is isomorphic to $S_{3}$, where an
automorphism $\Delta$ of order $3$ is diagonal, obtained via conjugation with
a suitable $\delta \in \PGL(3, p)$, and one of the
involutions is the transpose inverse automorphism $\top$.

Moreover, the Schur multiplier $M(T_{p})$ of $T_{p}$ has order $3$
\cite[3.3.6]{simple}, it is inverted
by $\top$, and clearly centralized by $\Delta$.

Let $P$ be  the natural $\F_{p}$-permutation module of  $T_{p}$ in its
action on the  points of the projective plane.  $P$  is the direct sum
of a copy of  the trivial module, and of a  module $N$.  The structure
of $N$  is investigated  in~\cite{ZaSup},~\cite{Abd}.   In particular,  it is
shown in~\cite{Abd} that $N$ has  a unique composition series $\Set{0}
\subset N_{1} \subset N$, such that  $N_{1}$ and $N_{2} = N/N_{1}$ are
dual to  each other, exchanged by  $\top$, but not isomorphic  to each
other.   Note that  $\Delta$  still  acts on  $P$,  as  it comes  from
conjugation with  an element  $\delta \in \PGL(3, p)$,  and thus also  acts on
$N_{1}$ and $N_{2}$.

Consider  the  natural  semidirect   product  $S_{p}$  of  $N_{1}$  by
$T_{p}$. It has been shown by K.I.~Tahara~\cite{Tahara} that the Schur
multiplier of $T_{p}$  is a direct summand of the  Schur multiplier of
$S_{p}$, that is, $M(S_{p}) = M(T_{p})  \oplus K$ for some $K$. We may
thus consider the central extension  $L_{p}$ of $M(T_{p})$ by $S_{p}$,
which is  the quotient of the  covering group of $S_{p}$  by $K$. Thus
$Z(L_{p}) =  M(T_{p})$ (Derek  Holt has shown  to us  calculations for
small primes, based on the description of~\cite{Tahara}, which appear to
indicate that actually $K = \Set{0}$ here.)

An automorphism $\sigma$ of $L_{p}$ induces automorphisms $\alpha$ of
$N_{1}$ and $\beta$ of $T_{p}$. Abusing notation slightly, we have,
for $n \in N_{1}$ and $h \in T_{p}$,
\begin{equation*}
  (n^{h})^{\alpha} 
  = 
  (n^{h})^{\sigma} 
  =
  (n^{\sigma})^{h^{\sigma}}
  =
  (n^{\alpha})^{h^{\beta}}.
\end{equation*}
If $\beta$ is an involution in $\Out(T_{p})$, say $\beta = \Delta^{-1}
\top \Delta$, then we have
\begin{equation*}
  h^{\beta} = \delta^{-2} h^{\top} \delta^{2},
\end{equation*}
so that
\begin{equation*}
  (n^{h})^{\alpha \Delta^{-2}} 
  =
  (n^{\alpha \Delta^{-2}})^{h^{\top}},
\end{equation*}
that is, $\alpha \Delta^{-2}$ is an isomorphism of $N_{1}$ with its
dual, a contradiction. Then 
$\Aut(L_{p})$ induces on $T_{p}$ only inner automorphisms and at most outer
automorphisms of order $3$, all of which centralize $M(T_{p}) = Z(L_{p})$.


\begin{thebibliography}{10}

\bibitem{Abd}
K.~S. Abdukhalikov, \emph{Modular permutation representations of {${\rm
  PSL}(n,p)$} and invariant lattices}, Mat. Sb. \textbf{188} (1997), no.~8,
  3--12. \MR{1481392}

\bibitem{Bid}
J.~N.~S. Bidwell, \emph{Automorphisms of direct products of finite groups.
  {II}}, Arch. Math. (Basel) \textbf{91} (2008), no.~2, 111--121. \MR{2430793}

\bibitem{Byo}
N.~P. Byott, \emph{Uniqueness of {H}opf {G}alois structure for separable field
  extensions}, Comm. Algebra \textbf{24} (1996), no.~10, 3217--3228.
  \MR{1402555}

\bibitem{ByoChi}
Nigel~P. Byott and Lindsay~N. Childs, \emph{Fixed-point free pairs of
  homomorphisms and nonabelian {H}opf-{G}alois structures}, New York J. Math.
  \textbf{18} (2012), 707--731. \MR{2991421}

\bibitem{fgag}
A.~Caranti and F.~Dalla~Volta, \emph{The multiple holomorph of a finitely
  generated abelian group}, J. Algebra \textbf{481} (2017), 327--347.
  \MR{3639478}

\bibitem{affine}
A.~Caranti, F.~Dalla~Volta, and M.~Sala, \emph{Abelian regular subgroups of the
  affine group and radical rings}, Publ. Math. Debrecen \textbf{69} (2006),
  no.~3, 297--308. \MR{2273982 (2007j:20001)}

\bibitem{CarChi}
Scott Carnahan and Lindsay Childs, \emph{Counting {H}opf {G}alois structures on
  non-abelian {G}alois field extensions}, J. Algebra \textbf{218} (1999),
  no.~1, 81--92. \MR{1704676 (2000e:12010)}

\bibitem{Chi}
Lindsay~N. Childs, \emph{On the {H}opf {G}alois theory for separable field
  extensions}, Comm. Algebra \textbf{17} (1989), no.~4, 809--825. \MR{990979}

\bibitem{FCC}
S.~C. Featherstonhaugh, A.~Caranti, and L.~N. Childs, \emph{Abelian {H}opf
  {G}alois structures on prime-power {G}alois field extensions}, Trans. Amer.
  Math. Soc. \textbf{364} (2012), no.~7, 3675--3684. \MR{2901229}

\bibitem{GrePar}
Cornelius Greither and Bodo Pareigis, \emph{Hopf {G}alois theory for separable
  field extensions}, J. Algebra \textbf{106} (1987), no.~1, 239--258.
  \MR{878476}

\bibitem{GroQiao}
Joshua~A. Grochow and Youming Qiao, \emph{Algorithms for group isomorphism via
  group extensions and cohomology}, 2013, arXiv:1309.1776v1 [cs.DS].

\bibitem{gruen}
Otto Gr{\"u}n, \emph{Beitr\"age zur {Gruppentheorie}, {I}}, J. Reine Angew.
  Math. \textbf{174} (1935), 1--14.

\bibitem{bra}
L.~Guarnieri and L.~Vendramin, \emph{Skew braces and the {Y}ang-{B}axter
  equation}, Math. Comp. \textbf{86} (2017), no.~307, 2519--2534. \MR{3647970}

\bibitem{Kohl-twisted}
Timothy Kohl, \emph{Groups of order {$4p$}, twisted wreath products and
  {H}opf-{G}alois theory}, J. Algebra \textbf{314} (2007), no.~1, 42--74.
  \MR{2331752 (2008e:12001)}

\bibitem{Kohl-multi}
\bysame, \emph{Multiple holomorphs of dihedral and quaternionic groups}, Comm.
  Algebra \textbf{43} (2015), no.~10, 4290--4304. \MR{3366576}

\bibitem{Miller-multi}
G.~A. Miller, \emph{On the multiple holomorphs of a group}, Math. Ann.
  \textbf{66} (1908), no.~1, 133--142. \MR{1511494}

\bibitem{Mills-multi}
W.~H. Mills, \emph{Multiple holomorphs of finitely generated abelian groups},
  Trans. Amer. Math. Soc. \textbf{71} (1951), 379--392. \MR{0045117 (13,530a)}

\bibitem{Rob1996}
Derek J.~S. Robinson, \emph{A course in the theory of groups}, second ed.,
  Graduate Texts in Mathematics, vol.~80, Springer-Verlag, New York, 1996.
  \MR{1357169 (96f:20001)}

\bibitem{Tahara}
Ken~Ichi Tahara, \emph{On the second cohomology groups of semidirect products},
  Math. Z. \textbf{129} (1972), 365--379. \MR{0313417}

\bibitem{simple}
Robert~A. Wilson, \emph{The finite simple groups}, Graduate Texts in
  Mathematics, vol. 251, Springer-Verlag London, Ltd., London, 2009.
  \MR{2562037}

\bibitem{ZaSup}
A.~E. Zalesski{\u\i} and I.~D. Suprunenko, \emph{Permutation representations
  and a fragment of the decomposition matrix of symplectic and special linear
  groups over a finite field}, Sibirsk. Mat. Zh. \textbf{31} (1990), no.~5,
  46--60, 213. \MR{1088915}

\end{thebibliography}
\providecommand{\bysame}{\leavevmode\hbox to3em{\hrulefill}\thinspace}
\providecommand{\MR}{\relax\ifhmode\unskip\space\fi MR }
\providecommand{\MRhref}[2]{%
  \href{http://www.ams.org/mathscinet-getitem?mr=#1}{#2}
}
\providecommand{\href}[2]{#2}

\end{document}